\documentclass[a4paper,11pt,oneside]{amsart}
\usepackage{latexsym}
\usepackage{amsmath}
\usepackage{amsfonts}
\usepackage{amsthm}
\usepackage{amssymb}
\usepackage{fncylab}
\usepackage[utf8]{inputenc}

\newcommand{\F}{{{\mathcal F}}}
\newcommand{\U}{{{\mathcal U}}}

\newcommand{\pw}{{{\mathcal P}}}

\newcommand{\pomega}{\pw(\omega)}

 \swapnumbers

 \newtheorem*{theorem*}{Theorem}
 \newtheorem*{observation*}{Observation}
 \newtheorem{theorem}[subsection]{Theorem}
 \newtheorem{proposition}[subsection]{Proposition}
 \newtheorem{corollary}[subsection]{Corollary}
 \newtheorem{observation}[subsection]{Observation}

 \newtheorem{question}[subsection]{Question}
 \theoremstyle{definition}
 \newtheorem{definition}[subsection]{Definition}
 \newtheorem*{definition*}{Definition}
 \newtheorem*{question*}{Question}

 \newtheorem*{note}{Note}
 \newtheorem*{ack}{Acknowledgements}
 
\begin{document}
\title{Filter convergence in $\beta\omega$}

\author{Jonathan Verner}
\address{Department of Logic, Charles University\\
Palachovo nám. 2\\ 116 38 Praha 1, Czech Republic}
\email{jonathan.verner{@}ff.cuni.cz}


\begin{abstract}
 We give a combinatorial necessary condition on a filter $\F$ which admits injective $\F$-convergent sequences in $\beta\omega$.
 We also show that no analytic filter $\F$ admits an injective $\F$-convergent sequence in $\beta\omega$. This answers a question of
 T. Banakh, V. Mychaylyuk and L. Zdomskyy.
\end{abstract}
\maketitle

\section*{Introduction}

It is well known that $\beta\omega$, the Čech-Stone compactification of the natural numbers, cannot contain convergent sequences.
On the other hand, if we generalize the notion of convergence to convergence
with respect to a filter\footnote{Unfamiliar concepts used in this introduction will be defined below.}, we have the following easy observation.

\begin{observation*} If $X$ is a compact space and $\mathcal U$ is an ultrafilter, then any sequence in $X$ converges with respect to $\mathcal U$.
\end{observation*}

Ultrafilters are very far from the Fréchét filter, which corresponds to the standard convergence, so the above observation is not very
surprising. What is surprising, however, is that one can get very close --- in the sense of category --- to having convergent sequences in every compact space. 
More precisely, the following (see \cite{BMZ11}) is true:

\begin{theorem*}[Banakh, Mychaylyuk, Zdomskyy] Any infinite compact space contains an injective sequence which converges with respect to some meager filter.
\end{theorem*}

In the cited paper the authors ask about the borderline between filters $\F$, which admit an injective $\F$-convergent sequence in $\beta\omega$ and those
that don't. In particular, they ask, whether there is some analytic filter and an injective sequence in $\beta\omega$ which would converge with respect to this filter.
Below we give a necessary condition for a filter $\F$ to admit an injective $\F$-convergent sequence in $\beta\omega$. As a corollary we give
a negative answer to the above mentioned question by showing that no analytic filter admits a convergent sequence in $\beta\omega$. 

\section{Definitions}

In this section we recall some basic definitions and facts. We first start by generalizing the notion of a convergent sequence.

\begin{definition} If $\bar{x}=\langle x_n:n<\omega\rangle\subseteq X$ is a sequence, we say that it is \emph{injective}, provided $x_n\neq x_k$ for $n\neq k$.
Given a filter $\F$ on $\omega$ and $x\in X$ we write
\begin{displaymath}
 x = \F-\lim\bar{x}
\end{displaymath}
if for every neighbourhood $U$ of $x$ the set
\begin{displaymath}
 \{ n:x_n\in U\}
\end{displaymath}
is an element of $\F$. We say that $x$ is an \emph{$\F$-limit} of the sequence $\bar{x}$. If some injective sequence $\bar{x}\subseteq X$ has an $\F$-limit in $X$,
we say that \emph{$\F$ admits an $\F$-convergent sequence in $X$}.
\end{definition}

The notion of a filter limit in the case when $\F$ is an ultrafilter was introduced in \cite{Bernstein1970} in the context of nonstandard analysis,
the same notion for filters was considered in \cite{Furstenberg1981,Akin1997} in the context of dynamical systems. 

\begin{observation} If $x = \F-\lim\bar{x}$ and $A\subseteq\omega$ is a \emph{positive} (with respect to $\F$) set of indices, i.e. $A\in\F^{+}=\{X\subseteq\omega:\omega\setminus X\not\in\F\}$,
then $x\in\overline{\{x_n:n\in A\}}$.
\end{observation}

We now turn to filters (and ultrafilters) on $\omega$.

\begin{definition} Recall, that a filter on $\omega$ can be identified with a subset of the Cantor space $\pomega$. 
We thus say that a filter is \emph{closed}, \emph{meager}, $F_\sigma$ or \emph{analytic} if it is such under this identification.
\end{definition}
 
\begin{definition}[\cite{Katetov68,Booth70} for ultrafilters] Given filters $\F,{\mathcal H}$ we say that $\F$ is \emph{Rudin-Keisler} above ${\mathcal H}$ if there is a function
$f:\omega\to\omega$ such that ${\mathcal H} = f_*(\F)=\{A:f^{-1}[A]\in\F\}$. In this case we write $\F\geq_{RK}{\mathcal H}$.
\end{definition}

\begin{proposition}\label{no_analytic_ultrafilters} No analytic filter is RK-above a nonprincipal ultrafilter.
\end{proposition}
\begin{proof} First notice that if $f:\omega\to\omega$ is a function, then the function $\bar{f}:2^\omega\to2^\omega$ given by $\bar{f}(A)=f^{-1}[A]$ is
continuous. Now suppose $\F$ is analytic and that $\F$ is RK-above some ultrafilter ${\mathcal U}$ as witnessed by $f$. Then, by definition, 
${\mathcal U}= \{A:\bar{f}(A)\in\F\} = \bar{f}^{-1}[\F]$ so ${\mathcal U}$ is a continuous preimage of an analytic set and hence is also analytic.
Since analytic sets have the Baire property (see e.g. \cite{Kechris}, Theorem 29.5) while nonprinciple ultrafilters 
cannot have the Baire property (see e.g. \cite{Kechris}, Exercise 8.50), this implies that $\mathcal U$ is principle.
\end{proof}

We will also need the following definition/fact about $\beta\omega$

\begin{definition} 
The space $\beta\omega$ can be identified with the set of ultrafilters on natural numbers. The base for the topology consists of (clopen) sets of the form 
$\hat{A}$ for $A\subseteq\omega$ where
\begin{displaymath}
 \hat{A} = \{ p\in\beta\omega: A\in p\}
\end{displaymath}
\end{definition}

%
%

\section{Main theorem}

\begin{theorem}\label{necessary_condition} If $\F$ is a filter such that for any partition $\langle A_n:n<\omega\rangle$ of $\omega$ into sets not in $\F$ there is a 
$B\subseteq\omega$ such that both $\bigcup_{n\in B}A_n$ and $\bigcup_{n\not\in B} A_n$ are $\F$-positive, then no injective $\bar{x}\subseteq\beta\omega$ has an $\F$-limit.
\end{theorem}
\begin{proof}
Suppose that for each partition $A_n$ of $\omega$ into sets not in $\F$ there is a $B\subseteq\omega$ so that both $\bigcup_{n\in B}A_n$ and $\bigcup_{n\not\in B}A_n$
are $\F$-positive. Aiming towards a contradiction assume there is some $\bar{x}$ and $p=\F-\lim\bar{x}$. Without loss of generality $p\not\in\bar{x}$
Recursively pick $y_n\in \bar{x}$ and pairwise disjoint basic clopen neighbourhoods $\hat{C_n}$ of $y_n$, with $C_n\subseteq\omega$,
such that no $\hat{C_n}$ contains $p$ and together they cover the sequence $\bar{x}$.
Let $A_n=\{k:x_k\in \hat{C_n}\}$. Then $A_n\not\in\F$, since $p\not\in\hat{C_n}$. By assumption we can find $B\subseteq\omega$,
such that both $\bigcup_{n\in B} A_n$ and $\bigcup_{n\not\in B}A_n$ are $\F$-positive. Since $p$ is an ultrafilter it contains either
$\bigcup_{n\in B}C_n$ or $\bigcup_{n\not\in B}C_n$. In the former case, this contradicts the fact that set $\bigcup_{n\not\in B}A_n$ is $\F$-positive,
in the latter case it contradicts the fact that the set $\bigcup_{n\in B}A_n$ is $\F$-positive.
\end{proof}

\begin{corollary} No analytic filter $\F$ admits an injective $\F$-convergent sequence in $\beta\omega$.
\end{corollary}
\begin{proof} Assume $\F$ admits an injective $\F$-convergent sequence in $\beta\omega$. By theorem \ref{necessary_condition}
there must be a partition $\langle A_n:n<\omega\rangle$ of $\omega$ into sets not in $\F$ such that for each
$B\subseteq\omega$ either $\bigcup_{n\in B}A_n\in\F$ or $\bigcup_{n\not\in B}A_n\in\F$. Let $g(k)=min\{n:k\in A_n\}$. Then
$g$ shows that $\F$ is RK-above the ultrafilter $\{B:\bigcup_{n\in B}A_n\in\F\}$ so, by proposition \ref{no_analytic_ultrafilters},
$\F$ cannot be analytic.
\end{proof}

\begin{note} Note that the fact we are talking about $\beta\omega$ and not just any space without a nontrivial convergent
sequence is crucial here. The following construction shows that there are spaces without nontrivial convergent sequences
which contain $\F$-convergent sequences even for some $F_{\sigma\delta}$ filters. 
\end{note}

\begin{proposition} If $\F$ is a $P$-filter properly extending the Fréchét filter then there is a compact space $X$ which contains
an injective $\F$-convergent sequence but does not contain convergent sequences.
\end{proposition}
\begin{proof} Let $X_{\F}$ be the quotient space $\beta\omega/H$, for $H=(\bigcap_{F\in\F} \overline{F}^{\beta\omega})\setminus\omega$ and let $z$
be the point corresponding to $H$. Then $X_{\F}$ is a compact space. We first show that $X_{\F}$ contains no injective convergent sequences.
Let $\langle x_n:n<\omega\rangle$ be an injective sequence in $X_{\F}$ and without loss of generality assume $x_n\neq z$. Notice that the sequence 
cannot converge to a point in $X_{\F}\setminus\{z\}$ (otherwise it would be a convergent sequence in $\beta\omega$). So, aiming towards a contradiction, 
assume $z=\lim_{n\to\infty} x_n$. Since $x_n\not\in H$ we can choose $A_n\in\F$ such that $A_n\not\in x_n$. Since $\F$ is a P-filter, we can
find a pseudointersection $A\in\F$, i.e. $A\subseteq^* A_n$ for all $n<\omega$. Then $A$ is a neighbourhood of $z$ which misses each $x_n$ --- a contradiction.
So it remains to show that $z=\F-\lim n$ (recall that $\omega\subseteq X_{\F}$), but this is obvious.
\end{proof}

\begin{note} This construction only works for $P$-filters. Indeed, if $\langle A_n:n<\omega\rangle$ are infinite disjoint subsets of $\omega$
and $\F$ is the filter generated by cofinite sets and sets of the form $\{B:(\forall^\infty n)(A_n\subseteq B)\}$, then in $X_{\F}$
any sequence $\langle x_n:n<\omega\rangle$ such that $A_n\in x_n$ will converge to $z$. This motivates the following question
\end{note}

\begin{question} Let $\F$ be a filter properly extending the Fréchét filter. Can we always find a compact space $X$ containing an injective $\F$-convergent
sequence but no convergent sequence?
\end{question}

After reading the proof of theorem 2 in \cite{BMZ11} one might well be tempted to ask:

\begin{question} If $\F$ is RK-above a nonprincipal ultrafilter $\U$, does it admint an injective $\F$-convergent sequence in
every compact space? Does it, at least, admit an injective $\F$-convergent sequence in $\beta\omega$.
\end{question}

We conjecture the answer to this question is yes however all our attempts at proving this turned out to lead nowhere.

\begin{ack} I would like to thank Taras Banakh for discussing the original problem with me and for encouraging me to work on it and K.~P.~Hart for reading
the paper and pointing out mistakes.
\end{ack}

\bibliographystyle{amsplain}
\bibliography{f-convergence}

\end{document}